\documentclass[11pt,a4paper]{amsart}
\usepackage[english]{babel}
\usepackage{amsfonts,amsmath}	
\usepackage{epsfig}
\usepackage{graphicx}		
\usepackage{amsthm,amssymb}
\usepackage[subnum]{cases}
\usepackage{url}
\usepackage{verbatim}
\usepackage{texdraw}
\usepackage{cite}
\usepackage{color}

\newtheorem{thm}{Theorem}
\newtheorem{cor}[thm]{Corollary}
\newtheorem{lem}[thm]{Lemma}
\newtheorem{defi}[thm]{Definition}

\newtheorem{exam}[thm]{Example}
\newtheorem{rem}[thm]{Remark}
\newtheorem{claim}[thm]{Claim}

\newcommand{\RR}{\mathbb{R}}

\newcommand{\Div}{\text{Div}}
\newcommand{\rk}{\text{rk}}
\newcommand{\outdeg}{\text{outdeg}}
\newcommand{\supp}{\text{supp}}

\newcommand{\PP}{\mathbb{P} }

\newcommand{\LL}{\mathcal{L}\,}
\newcommand{\OO}{\mathcal{O} }

\newcommand{\II}{\mathcal{I}\,}

\title[\resizebox{4.7in}{!}{Brill--Noether theory of curves on $\PP^1 \times \PP^1$: tropical and classical approach}]{Brill--Noether theory of curves on $\PP^1 \times \PP^1$: tropical and classical approach}
\author{Filip Cools}
\address{Filip Cools, KU Leuven, Department of Mathematics, Celestijnenlaan 200B, B-3001 Heverlee, Belgium}
\email{{\tt f.cools@kuleuven.be}}
\author{Michele D'Adderio}
\address{Michele D'Adderio, Universit\'e Libre de Bruxelles (ULB), D\'epartement de Math\'e\-matique, Boulevard du Triomphe, B-1050 Bruxelles, Belgium}
\email{{\tt mdadderi@ulb.ac.be}}
\author{David Jensen}
\address{Dave Jensen, University of Kentucky, Department of Mathematics, 719 Patterson Office Tower Lexington, KY 40506-0027, USA}
\email{{\tt dave.h.jensen@gmail.com}}
\author{Marta Panizzut}
\address{Marta Panizzut, TU Berlin, Institut f\"ur Mathematik, Stra{\ss}e des 17.Juni 136, 10623 Berlin, Germany}
\email{{\tt panizzut@math.tu-berlin@de}}
\thanks{Marta Panizzut is supported by the Einstein Foundation Berlin.}

\begin{document}
\begin{abstract} The gonality sequence $(d_r)_{r\geq1}$ of a smooth algebraic curve comprises the minimal degrees $d_r$ of linear systems of rank $r$. We explain two approaches to compute the gonality sequence of smooth curves in $\mathbb{P}^1 \times \mathbb{P}^1$: a tropical and a classical approach. The tropical approach uses the recently developed Brill--Noether theory on tropical curves and Baker's specialization of linear systems from curves to metric graphs \cite{Bak}. The classical one extends the work \cite{Har} of Hartshorne on plane curves to curves on $\mathbb{P}^1 \times \mathbb{P}^1$.
\end{abstract}

\maketitle

\section{Introduction}

Let $C$ be a smooth irreducible projective curve of genus $g\geq 4$ over an algebraically closed field $k$ of characteristic zero.
We will denote a (complete) linear system $|D|$ of divisors on $C$ with rank $r=\rk_C(D)=H^0(C,\mathcal{O}_C(D))-1$ and degree $d=\deg(D)$ with
$g_d^r$. The \emph{gonality sequence} $(d_r(C))_{r\geq 1}$ of $C$ was introduced in \cite{LM} and is defined as follows:
$$d_r(C)=\min\{d\in\mathbb{Z}\,|\,C\ \text{admits a linear system}\ g_d^r\}.$$
Alternatively, $d_r(C)$ equals the smallest \emph{degree} of a non-degenerate rational map $f:C\to \mathbb{P}^r$.
Hereby, the degree of $f$ is defined as the degree of $f$ onto its image times the degree of the image curve $\overline{f(C)}\subset \mathbb{P}^r$. The first entry $d_1(C)$ of the gonality sequence is called the \emph{gonality} of $C$.
Because of the Riemann-Roch Theorem, we have that $d_r(C)=g+r$ if $r\geq g$. Hence, the entries $d_r(C)$ of interest are the ones with index $0<r<g$, which correspond to special divisors $D$ on $C$.

The whole gonality sequence is determined for general, hyperelliptic, trigonal, bielliptic, general tetragonal and general pentagonal curves (see \cite{LM,LN}). It is also known for smooth plane curves \cite{Cil,Har}.

\begin{thm} \label{thm_planecurves}
Let $C\subset \mathbb{P}^2$ be a smooth plane curve of degree $d$. Let $r$ be an integer satisfying $0<r<g:=\frac{(d-1)(d-2)}{2}$.
Then $$d_r(C)=kd-h,$$ where $k$ and $h$ are the uniquely determined integers with $1\leq k \leq d-3$ and $0\leq h\leq k$ such that
$r=\frac{k(k+3)}{2}-h$.
\end{thm}

In particular, the linear systems on $C\subset\mathbb{P}^2$ with minimal degree for a certain fixed rank are fully classified: they are the ones which naturally come from the plane embedding of $C$, i.e. they are cut out by plane curves of some other fixed degree, minus some assigned base points.
\vspace{\baselineskip}

In this article, we compute the gonality sequence of smooth curves $C$ on $\PP^1 \times \PP^1$, extending the result on plane curves. In order to state the main result, we introduce some notations: for $m,n\in\mathbb{Z}_{>0}$ and $0<r<g$ with $g=(m-1)(n-1)$, let $I_r$ be the set of triples
$(a,b,h)\in\mathbb{Z}^3$ satisfying
$$0\leq a\leq m-1\ ,\ 0\leq b\leq n-1\ ,\ h\geq 0 \text{ and } r=(a+1)(b+1)-1-h,$$ and define
$$\delta_r(m,n)=\min\{a n+b m-h\,|\,(a,b,h)\in I_r\}.$$

\begin{thm} \label{thm_main}
Let $C$ be a smooth curve of bidegree $(m,n)$ on $\mathbb{P}^1\times \mathbb{P}^1$. Then $d_r(C)=\delta_r(m,n)$ for all $0<r<g(C)=(m-1)(n-1)$.
\end{thm}

\begin{rem}
In \cite{LM}, given a smooth curve with gonality sequence $(d_r)_{r}$, the authors investigate whether the \emph{slope inequality} $\frac{d_r}{r}\geq \frac{d_{r+1}}{r+1}$ is satisfied. Although for most curves this inequality is valid everywhere, there are counterexamples of curves where the inequality is violated for some rank values. The latter is also the case for smooth curves on $\mathbb{P}^1\times \mathbb{P}^1$. For example, if we fix the bidegree $(m,n)=(7,5)$, then the slope inequality is violated at $r=5$ (since $d_5=17$ and $d_6=21$) and at $r=11$ (since $d_{11}=29$ and $d_{12}=32$).
\end{rem}

We will present two ways to attack the problem: a combinatorial and an algebro-geometric way. The tropical approach relies on the theory of linear systems on metric graphs/tropical curves, which has been introduced in \cite{Bak,GK,MZ}. Using this theory, we can also introduce the gonality sequence $(d_r(\Gamma))_{r\geq 1}$ for metric graphs $\Gamma$:
\[d_r(\Gamma)=\min\{d\in\mathbb{Z}\,|\,\exists D\in\text{Div}(\Gamma): \deg(D)=d\ \text{and}\ \rk_{\Gamma}(D)\geq r\}.\]
If the metric graph $\Gamma$ has genus $g$, we again have that $d_r(\Gamma)=g+r$ for $r\geq g$, because of the Riemann--Roch Theorem for metric graphs (see \cite{BaNo,GK,MZ}). Here, we will focus on the metric complete bipartite graph $K_{m,n}$ with all edge lengths equal to one.

\begin{thm} \label{thm_Kmn}
For $0<r<g(K_{m,n})=(m-1)(n-1)$, we have that
$$d_r(K_{m,n})=\delta_r(m,n).$$
\end{thm}

Using a degeneration of bidegree-$(m,n)$ curves to a union of $m+n$ lines and Baker's Specialization Lemma \cite[Lemma 2.8]{Bak}, we prove Theorem \ref{thm_main} for \emph{generic} bidegree-$(m,n)$ curves (see Corollary \ref{cor_generic} for the details).

To settle Theorem \ref{thm_main} for \emph{arbitrary} smooth bidegree-$(m,n)$ curves, we make use of the notion of \emph{generalized divisors} on Gorenstein curves, which has been developed in \cite{Har} by Hartshorne to fix Noether's incomplete proof \cite{Noe} of Theorem \ref{thm_planecurves}. We adapt Hartshorne's argument to curves on $\mathbb{P}^1\times\mathbb{P}^1$.
\vspace{\baselineskip}

The setup of the paper is as follows. In Section 2, we show that the minimum formula $\delta_r(m,n)$ is an upper bound for the entry $d_r(C)$ of the gonality sequence of a bidegree-$(m,n)$ curve $C$. In Section 3, we provide a characterization for reduced divisors on metric graphs, which we believe to be of interest on its own. Theorem \ref{thm_Kmn} is proven in Section 4.
Finally, in Section 5, we establish Theorem \ref{thm_main} using Hartshorne's methods.  This final section can be read independently of the previous sections.


\section{The upper bound for smooth curves on $\PP^1 \times \PP^1$}

Let $k$ be an algebraically closed field of characteristic zero and denote $S=\mathbb{P}^1\times \mathbb{P}^1$. We recall that $S$ has two line rulings:
$$R_1=\{\{P\}\times \mathbb{P}^1\,|\,P\in \mathbb{P}^1\} \ \text{and}\ R_2=\{\mathbb{P}^1\times \{Q\}\,|\,Q\in \mathbb{P}^1\}.$$
If $L_1\in R_1$ and $L_2\in R_2$, then $\text{Pic}(S)=\mathbb{Z}[L_1]+\mathbb{Z}[L_2]$, with intersection numbers given by $L_1\cdot L_1=L_2\cdot L_2=0$ and $L_1\cdot L_2=1$. The anti-canonical class of $S$ is equal to $2[L_1]+2[L_2]$.

Consider a smooth curve $C\subset S$ of bidegree $(m,n)$, i.e. $$[C]=m[L_1]+n[L_2].$$ Note that $C$ has genus $g(C)=(m-1)(n-1)$.
For $i\in\{1,2\}$, write $D_i$ to denote the divisor on $C$ cut out by $L_i$. Hence, we have that $\deg(D_1)=n$, $\deg(D_2)=m$ and $\rk_C(D_1)=\rk_C(D_2)=1$. In fact, the gonality $d_1(C)$ of $C$ is equal to $\min\{m,n\}$ (see e.g. \cite[Corollary 6.2]{CaCo}).

\begin{rem} \label{rem_phi}
Consider the map $$\varphi:k^2\to \mathbb{P}^3:(x,y)\to (1:x:y:xy).$$ Then $S=\overline{\text{im}(\varphi)}$, the ruling $R_1$ consists of the lines on $S$ with fixed $x$-value and the ruling $R_2$ of lines with fixed $y$-value. The curve $C$ corresponds to (the zero set of) a bivariate polynomial $f$ of the form $\sum_{i=0}^m\sum_{j=0}^n a_{i,j} x^i y^j$, where the coefficients $a_{i,j}\in k$.
\end{rem}

\begin{lem}  \label{lem_divisors}
If $a,b\in\mathbb{Z}$ with $0\leq a\leq m-1$ and $0\leq b\leq n-1$, then $$\rk_C(a D_1+b D_2)=(a+1)(b+1)-1.$$
\end{lem}

\begin{proof}
Consider the exact sequence of sheaves of $\mathcal{O}_S$-modules $$0\to \mathcal{O}_S(-C)\to \mathcal{O}_S\to \mathcal{O}_C\to 0.$$
Since $C\sim m L_1+n L_2$, we also have $$0\to \mathcal{O}_S(-m L_1-n L_2)\to \mathcal{O}_S\to \mathcal{O}_C\to 0.$$
We tensor the former exact sequence with $\mathcal{O}_S(a L_1+b L_2)$ to obtain
$$0\to \mathcal{O}_S(-(m-a) L_1-(n-b) L_2)\to \mathcal{O}_S(a L_1+b L_2)\to \mathcal{O}_C(a D_1+b D_2)\to 0.$$
Taking cohomology gives us the long exact sequence
\begin{multline*}
0\to H^0(S,-(m-a) L_1-(n-b) L_2)\to H^0(S,a L_1+b L_2) \\ \to H^0(C,a D_1+b D_2)\to H^1(S,-(m-a) L_1-(n-b) L_2) \to \cdots ,$$
\end{multline*}
where we have abbreviated cohomology spaces $H^i(X,\mathcal{O}_X(D))$ by $H^i(X,D)$.
Using \cite[Prop.4.3.3]{CLS}, we have that $$h^0(S,-(m-a) L_1-(n-b) L_2)=0\ \text{and}\ h^0(S,a L_1+b L_2)=(a+1)(b+1).$$ Hence, it suffices to check that $H^1(S,-(m-a) L_1-(n-b) L_2)=0$, which follows from the Batyrev-Borisov Vanishing Theorem.
\end{proof}

\begin{rem}
In fact, there is an appropriate divisor $D\sim a D_1+b D_2$ such that $$H^0(C,D)=\langle x^i y^j\,|\,i=0,\ldots,a; j=0,\ldots,b\rangle.$$
Here, $x$ and $y$ are viewed as functions on $C$ through the map $\varphi$ defined in Remark \ref{rem_phi}. In particular, by the adjunction formula $K_C=(K_S+C)|_C$, we have that $K_C\sim (m-2)D_1+(n-2)D_2$. So there is a canonical divisor on $C$ such that
$$H^0(C,K_C)=\langle x^i y^j\,|\,i=0,\ldots,m-2; j=0,\ldots,n-2\rangle.$$
\end{rem}

In what follows, we use the notations $I_r$ and $\delta_r(m,n)$, which have been introduced in Section 1.

\begin{lem} \label{lem_upperbound}
If $0<r<g$, then $d_r(C)\leq \delta_r(m,n)$.
\end{lem}

\begin{proof}
Fix an element $(a,b,h)\in I_r$ that attains the minimum in the formula for $\delta_r(m,n)$, i.e. $\delta_r(m,n)=a n+b m-h$. Let $D$ be a divisor on $C$ of the form $a D_1+b D_2-E$, where $E$ is an effective divisor of degree $h$. Using Lemma 1, we have that $$\rk_C(D)\geq \rk_C(a D_1+b D_2)-h=(a+1)(b+1)-1-h=r$$ (if $E$ is generic, then in fact equality holds), hence $$d_r(C)\leq \deg(D)=a n+b m-h=\delta_r(m,n).$$
\end{proof}

\begin{rem} \label{condition}
The condition that the triple $(a,b,h) \in I_r$ attains the minimum in the formula for $\delta_r(m,n)$, i.e. $\delta_r(m,n)=a n+b m-h$, implies that $(a,b)$ attains the maximum of the set
\[\big\{(m-a-1)(n-b-1)\,\big|\, h:= (a+1)(b+1) -1-r \geq 0\big\}.\]
\end{rem}

\begin{lem}
Consider the subset $I'_r\subset \mathbb{Z}^3$ of triples $(a,b,h)$ satisfying $0\leq a\leq m-2$, $0\leq b\leq n-2$, $0\leq h\leq \min\{a,b\}$ and $r=(a+1)(b+1)-1-h$. Then $$\delta_r(m,n)=\min\{a n+b m-h\,|\,(a,b,h)\in I'_r\},$$ hence the minimum formula for $\delta_r(m,n)$ is already attained on a strictly smaller subset $I'_r\subset I_r$.
\end{lem}

\begin{proof}
Firstly, it suffices to restrict to the cases $0\leq a<m-1$ and $0\leq b<n-1$. Indeed, the maximum in Remark \ref{condition} is at least $1$ since $(a,b)=(m-2,n-2)$ satisfies the condition on $h$ for all $0<r<g$, hence $a\neq m-1$ and $b\neq n-1$. There is a more geometric reason for this: if $0<r<g$ and a divisor $D$ has rank $r$ and degree $d=d_r(C)$, then $D$ will have to be special (because of Riemann-Roch), hence it is contained in a canonical divisor $K_C\sim (m-2)D_1+(n-2)D_2$.

Moreover, it suffices to consider $0\leq h\leq \min\{a,b\}$. Indeed, assume for instance that $h\geq a+1$. If $b\geq 1$, then we can replace $b$ by $b-1$ and $h$ by $h-a-1$. This does not change the rank $r=(a+1)(b+1)-1-h$, but the degree $d$ decreases by $m-a-1>0$. If $b=0$, then $r=a-h<0$, a contradiction.
\end{proof}

To end this section, we return to the alternative definition of the gonality sequence $(d_r(C))_r$ as the smallest degrees of non-degenerate rational maps $f:C\to \mathbb{P}^r$. Let $C\subset S$ be a bidegree-$(m,n)$ curve. Although we did not prove Theorem \ref{thm_main} yet, it will turn out that the linear systems $g_d^r=|aD_1+bD_2-E|$ on $C$, with $(a,b,h)\in I_r$ attaining the minimum $\delta_r(m,n)=an+bm-h$, are of smallest degree. These linear systems correspond to tangible rational maps $f:C\to \mathbb{P}^r$. Let's give an example.

\begin{exam}
Take $(m,n)=(5,4)$, so $C$ has genus $g=12$. Below, we present the maps with smallest degree to $\mathbb{P}^r$ for $r\in\{1,2,3\}$. Hereby, we use the embedding $C\subset S\subset \mathbb{P}^3$ from Remark \ref{rem_phi}.
\begin{itemize}
\item For $r=1$, the minimum $d_1(C)=4$ is attained by $(1,0,0)\in I_1$. This triplet corresponds to the map
$$f_1:C\to\mathbb{P}^1: (1:x:y:xy)\mapsto (1:x).$$
\item For $r=2$, both $(2,0,0)$ and $(1,1,1)$ in $I_2$ attain the minimum $d_2(C)=8$. The first triplet corresponds to
$$f_2:C\to\mathbb{P}^2:(1:x:y:xy)\mapsto (1:x:x^2)$$ while the second one corresponds to maps of the form
$$f'_2:C\to\mathbb{P}^2:(1:x:y:xy)\mapsto (x-x':y-y':(x-x')(y-y'))$$ where $P=\varphi(x',y')\in C$.
\item For $r=3$, the minimum $d_3(C)=9$ is attained by $(1,1,0)\in I_3$ which corresponds to the identity map $$f_3:C\to\mathbb{P}^3:(1:x:y:xy)\mapsto (1:x:y:xy).$$
\end{itemize}
\end{exam}

\section{Characterization of reduced divisors on metric graphs}

In the following two sections, we will use the theory of linear systems of divisors on metric graphs/tropical curves. For the definitions, we refer to \cite{Bak,BJ,GK,MZ}. We will often use the terminology of chip firing (see e.g. \cite[Remark 2.2]{CDPR}).

The notion of reduced divisors on metric graphs will play an important role in the proof of the theorem. We begin by recalling the definition given in \cite{HKN, Luo}. Afterwards, we provide a new characterization of reduced divisors on metric graphs with arbitrary edge lengths.

\begin{defi}
Let $\Gamma$ be a metric graph and $X$ be a closed connected subset of $\Gamma$. Given $p \in \partial X$, the \emph{outgoing degree $\outdeg_X(p)$ of $X$ at $p$} is defined as the maximum number of internally disjoint segments in $\Gamma \setminus X$ with an open end in $p$. Let $D$ be a divisor on $\Gamma$. A boundary point $p \in \partial X$ is \emph{saturated with respect to $X$ and $D$} if $D(p) \geq \outdeg_X(p)$, and \emph{non-saturated} otherwise. A divisor $D$ is \emph{$p$-reduced} if it is effective in $\Gamma \setminus \{p\}$ and each closed connected subset $X \subseteq \Gamma \setminus \{ p \}$ contains a non-saturated boundary point.
\end{defi}

\begin{thm} \label{reduced}
Let $\Gamma$ be a metric graph with vertex set $V(\Gamma)$ (containing the topological vertices) and edge set $E(\Gamma)$.
Let $v\in V(\Gamma)$ and $D\in \text{Div}(\Gamma)$. Then $D$ is $v$-reduced if and only if the following conditions are satisfied:
\begin{enumerate}
\item $D$ is nonnegative on $\Gamma\setminus\{v\}$;
\item for all edges $e\in E(\Gamma)$, we have that $\sum_{p\in e^\circ}\,D(p) \leq 1$;
\item there is a total order $\prec$ on $V(\Gamma)$
$$v=v_0\prec v_1 \prec \ldots \prec v_r$$ such that $\tilde{D}(v_i)<\tilde{d}(v_i)$ for all $v_i\in V(\Gamma)\setminus\{v\}$, where
$$\tilde{D}(v_i):=D(v_i)+\sum_{\ell<i} \sum_{e=(v_\ell,v_i)} \sum_{P\in e^\circ}\,D(p) \quad \text{and}$$
$$\tilde{d}(v_i):=\sharp\{e=(v_i,v_\ell)\in E(\Gamma)\,|\,\ell<i\};$$
\item moreover, for all $v_i,v_j\in V(\Gamma)\setminus\{v\}$: if $v_i \prec v_j$ and
$$D(v_j)+\sum_{\ell<i} \sum_{e=(v_j,v_\ell)} \sum_{p\in e^\circ}\,D(p) \ < \ \sharp\{e=(v_\ell,v_j)\in E(\Gamma)\,|\,\ell<i\},$$
then $$\tilde{D}(v_i)\leq D(v_j)+\sum_{\ell<i} \sum_{e=(v_j,v_\ell)} \sum_{p\in e^\circ}\,D(p).$$
(The right hand side is smaller than or equal to $\tilde{D}(v_j)$.
)
\end{enumerate}
\end{thm}

\begin{rem}
The order $\prec$ in $(3)$ is not unique. The condition (4) can be omitted, but we add it to limit the freedom of choice of the order $\prec$.
\end{rem}

\begin{proof}
First, let's prove the `only if' implication. So assume that $D$ is $v$-reduced. The condition $(1)$ follows from the definition of reducedness. Condition $(2)$ is also true: if $\sum_{p\in e^\circ}\,D(p) > 1$, we can construct a subset $X\subset e^\circ$ such that $\outdeg_X(p)\geq D(p)$ for all $p\in\partial X$ (so $X$ does not burn). More precisely, if there exist two different points $q,q'$ on $e^\circ$ with $D(q),D(q')\geq 1$, take $X=[q,q']$; otherwise take $X=\{q\}$ with $D(q)\geq 2$. For finding a total order on the vertices that satisfies $(3)$ and $(4)$, we use Dhar's Burning Algorithm (see \cite[Algorithm 2.5]{Luo}). We start the fire at the sink $v=v_0$. By reducedness, the whole graph has to burn, so there should be a neighbor of $v$, say $v_1$, that burns. This means that $$D(v_1)+\sum_{e=(v_0,v_1)} \sum_{p\in e^\circ}\, D(p) < \sharp\{e=(v_0,v_1)\in E(\Gamma)\}.$$ We can choose $v_1$ in such a way that it satisfies the above inequality and that the left hand side is minimal. We define the total order on $\{v_0,v_1\}$ by $v_1 \prec v_0$, hence both $(3)$ and $(4)$ are satisfied on $\{v_0,v_1\}$. We proceed by induction on the number of vertices on which the order $\prec$ is already defined, so assume that $v_0 \prec v_1 \prec \ldots \prec v_{i-1}$ and that the conditions $(3)$ and $(4)$ are satisfied. For $v_i$, we take a vertex that burns next, so if we take $v_{i-1} \prec v_i$, then $\tilde{D}(v_i)<\tilde{d}(v_i)$. Again, we can take $v_i\in V(\Gamma)\setminus \{v_0,\ldots,v_{i-1}\}$ such that $\tilde{D}(v_i)$ is minimal. Hence conditions $(3)$ and $(4)$ are satisfied on $\{v_0,\ldots,v_i\}$.

For the `if' implication, we have to show that the whole graph burns when we start the fire at $v$. We show that if the vertices $\{v_0,v_1,\ldots,v_{i-1}\}$ burn, then also $v_i$ and all the edges $e=(v_\ell,v_i)$ with $\ell<i$ burn. Since $\tilde{D}(v_i)<\tilde{d}(v_i)$ and $\sum_{p\in e^\circ}\,D(p) \leq 1$ for each edge $e=(v_\ell,v_i)$ with $\ell<i$, there are less chips in $v_i$ then there are edges $e=(v_\ell,v_i)$ with $\sum_{p\in e^\circ}\,D(p)=0$, so all these edges and $v_i$ burn. Now also the edges $e=(v_\ell,v_i)$ with $\ell<i$ and $\sum_{p\in e^\circ}\,D(p)=1$ burn, since the fire comes from two directions.
\end{proof}

\begin{exam}
If we take $\Gamma=K_d$, $V(\Gamma)=\{v_1,\ldots,v_d\}$ and $v=v_d$, we get back \cite[Lemma 17]{CP}: indeed, if we assume
$$v=v_d \prec v_1 \prec v_2 \prec \ldots \prec v_{d-1},$$
we have that $$\tilde{D}(v_i)< \tilde{d}(v_i)=i$$ for all $v_i\in V(\Gamma)\setminus\{v_d\}$. Condition $(4)$ implies that $$\tilde{D}(v_1)\leq \tilde{D}(v_2)\leq \ldots \leq \tilde{D}(v_{d-1}).$$
\end{exam}

\begin{exam}
If $\Gamma=K_{m,n}$ with $V(\Gamma)=\{v_1,v_2,\ldots,v_m,w_1,w_2,\ldots,w_n\}$ and $v=v_m$, after a relabeling of the other vertices, we can assume that $$v_m \prec v_1 \prec \ldots \prec v_{m-1} \quad \text{and} \quad v_m \prec w_1 \prec w_2 \prec \ldots \prec w_n.$$
Then $$\tilde{d}(v_j)=\sharp\{w_i\,|\,w_i \prec v_j\} \quad \text{and} \quad \tilde{d}(w_i)=\sharp\{v_j\,|\,v_j \prec w_i\}.$$
Condition $(4)$ implies that
$$\tilde{D}(v_1)\leq \tilde{D}(v_2)\leq \ldots \leq \tilde{D}(v_{m-1}) \quad \text{and} \quad \tilde{D}(w_1)\leq \tilde{D}(w_2)\leq \ldots \leq \tilde{D}(w_{n}).$$
Note that the conditions $(1)-(2)-(3)$ imply that $$r_i:=\tilde{D}(w_i)+1-\sharp\{j\in \{1,2,\dots,m-1\}\mid \tilde{D}(v_j)\leq i-2\}\leq 1.$$
(The $r$-vector $(r_1,\ldots,r_n)$ appears in \cite{DLB}.)
Indeed, if $v_j \prec w_i$ with $j\in\{1,\ldots,m-1\}$, then $$\tilde{D}(v_j)\leq \tilde{d}(v_j)-1 \leq i-2,$$ hence
\begin{align*}
\tilde{D}(w_i) &\leq \tilde{d}(w_i)-1 \\ &= \sharp\{j\in \{1,2,\dots,m-1\}\,|\,v_j \prec w_i\} \\
&\leq \sharp\{j\in \{1,2,\dots,m-1\}\mid \tilde{D}(v_j)\leq i-2\}.
\end{align*}
\end{exam}

\begin{rem}
If we define $\tilde{D}(v)=D(v)$ (and $\tilde{d}(v)=0$) for the vertex $v$ and $\tilde{D}=\sum_{v\in V(\Gamma)}\,\tilde{D}(v)$, then the divisors $D,\tilde{D}\in \Div(\Gamma)$ have the same degree, but $\tilde{D}$ is supported on the vertices.

One could wonder whether the following statement is valid: if $G$ is a regular graph (without loop edges) and $\Gamma$ is the corresponding metric graph where all the edge lengths are equal to one, then $\rk_\Gamma(D)\leq \rk_G(\tilde{D})$. By \cite[Theorem 1.3]{HKN}, we have that $\rk_G(\tilde{D})=\rk_\Gamma(\tilde{D})$. This statement would imply \cite[Conjecture 3.14]{Bak}, which predicts that the gonality sequences of $G$ and $\Gamma$ are equal. Unfortunately, there are counterexamples for the statement (personal communication with Jan Draisma and Alejandro Vargas).
\end{rem}

\section{The gonality sequence for metric complete bipartite graphs}

In this section, we focus on the metric complete bipartite graph $K_{m,n}$ with edge lengths $l(e)=1$, and prove Theorem
\ref{thm_Kmn}.

If $m=1$ or $n=1$, then $K_{m,n}$ is a tree and the theorem is obviously true.
So we fix integers $m,n>1$ and denote the topological vertices of $K_{m,n}$ by $$v_1, \dots, v_m, w_1, \dots, w_n.$$

\vspace{\baselineskip}

We begin by proving an upper bound for the gonality sequence $(d_r)_{r\geq1}$. To be precise, we show that the divisor $\sum_{i=1}^m b (v_i) + \sum_{i=1}^n a (w_i)$ on $K_{m,n}$ has degree $a n+b m$ and rank at least $(a+1)(b+1)-1$. Since for every point $p \in K_{m,n}$, we have that $\rk(D-p)\geq \rk(D)-1$, we see that the minimum formula $\delta_r(m,n)$ gives an upper bound for $d_r(K_{m,n})$.

In \cite[\S 2]{DLB}, the authors provide an algorithm to compute the rank of a divisor $D$ on the discrete complete bipartite graph $K_{m,n}$.
The algorithm takes a divisor as input. Step zero consists of computing the $v_m$-reduced divisor $D'$ equivalent to it. If $D'(v_m)<0$, then the divisor has rank $-1$. Otherwise the algorithm proceeds by taking a vertex $w_i$ such that $D'(w_i) = 0$ and by considering the divisor $D_1 = D - (w_i)$. Again, it computes the $v_m$-reduced divisor $D'_1$ linearly equivalent to $D_1$. If $D_1'(v_m) <0$, then the algorithm stops, otherwise it iterates. The algorithm terminates after at most $\deg(D)$ steps. The rank of the divisor $D$ is given by the number of steps minus one.

We use the algorithm to compute the rank of the divisor $D = \sum_{i=1}^m b (v_i) + \sum_{i=1}^n a (w_i)$ on the discrete complete bipartite graph. By \cite[Theorem 1.3]{HKN}, the divisor $D$ has the same rank on the complete bipartite graph $K_{m,n}$ with edge lengths equal to one.

\begin{thm}
The divisor $\sum_{i=1}^m b (v_i) + \sum_{i=1}^n a (w_i)$ on $K_{m,n}$ has degree $a n+b m$ and rank $(a+1)(b+1)-1$.
\end{thm}

\begin{proof} We write the subsequent divisors appearing in the algorithm as $D_{s,t}$ where $1\leq s \leq b+1$ and $t \geq 0$. In this way, we obtain a sequence of divisors:
\[ D_{1,0},\, D_{2,0}, \dots , \,  D_{b+1,0}, \,  D_{1,1}, \, D_{2,1}, \dots, \, D_{b+1,1},  \, D_{1,2}, \dots
\]
At every step, we subtract the divisor $(w_{i})$ corresponding to the vertex with zero coefficient and smallest index $i$.
The first two divisors are:
\[ D_{1,0} = D_{0,0} - (w_1) \sim   \sum_{i=1}^{m-1} (b-1) (v_i) + (b+an -1 )(v_m) + (m-1)(w_1)
	\]
	\[
	D_{2,0} = D_{1,0} - (w_2) \sim  \sum_{i=1}^{m-1} (b-2) (v_i) + (b+an -2 )(v_m) + \sum_{i=1}^2 (m-1)(w_i)
 \]
At the step $(s,t)$ with $s\leq b$ and $t \leq a$,  the divisor $D_{s,t}$ is linearly equivalent to the following $v_m$-reduced divisor:

\[
\sum_{i =1}^{m-1} (b- s) (v_i) + \big((a-t)n + b \big)(v_m) + \sum_{i=1}^{s} (m-1)(w_i) + \sum_{i=b+2}^n t(w_i).
\]
At the step $(b+1,t)$ with $t \leq a$, the divisor $D_{b+1,t}$ is linearly equivalent to the following $v_m$-reduced divisor:
\[
\sum_{i =1}^{m-1} b (v_i) + \big(b + (a-t-1)n \big)(v_m) + \sum_{i=b+2}^n(t+1)(w_i).
\]
Therefore, at step $(b,a)$, we get
\[
D_{b,a} \sim b(v_m)+ \sum_{i=1}^{b} (m-1)(w_i) + \sum_{i=b+2}^n a(w_i).
\]
At the following step, we obtain the divisor
\[
D_{b+1,a} \sim \sum_{i =1}^{m-1} b (v_i) + \big(b -n \big)(v_m) + \sum_{i=b+2}^n(a+1)(w_i),
\]
which is $v_m$-reduced and has negative coefficient at the vertex $v_m$. The algorithm terminates. Since we have made $(a+1)(b+1)$ steps in total, the rank is $(a+1)(b+1)-1$.
\end{proof}


We are left with showing that the minimum formula also provides a lower bound. So, if we take $(a,b,h)\in I_r$ so that it attains the minimum (hence $(a,b)$ attains the maximum in Remark \ref{condition}), then we need to show that each divisor $D$ with $\deg(D)\leq a n + b m - h - 1$ has $\rk(D)<r$.

Fix a divisor $D$ with $\deg(D)\leq a n + b m - h - 1$ and assume that it is reduced with respect to $v_m$. By Theorem \ref{reduced} there exists an ordering $\prec$ on the vertices such that $\tilde{D}(v)<\tilde{d}(v)$ for all $v\in V(K_{m,n})\setminus\{v_m\}$, where
$$\tilde{D}(v) =D(v)+  \sum_{v'\prec v} \sum_{p\in (v,v')^{\circ}}\,D(p) \qquad \text{and}$$
$$\tilde{d}(v) = \sharp\{e=(v,v')\in E(K_{m,n})\,|\,v'\prec v\}.$$ By relabeling if necessary, we may assume that
$$v_m \prec v_1 \prec \dots \prec v_{m-1} \ \text{and}\ w_1 \prec \dots \prec w_n.$$

\begin{lem}
The ordering $\prec$ can be chosen such that $\tilde{D}(v_i)=\tilde{d}(v_i)-1$.
\end{lem}

\begin{proof}
Suppose that a vertex $v_i$ does not satisfy $\tilde{D}(v_i)=\tilde{d}(v_i)-1$, hence $\tilde{D}(v_i)<\tilde{d}(v_i)-1$. We may assume that $v_i$ comes just after a vertex $w_j$ in the order $\prec$. Indeed, since all the vertices $v$ between $w_j$ and $w_{j+1}$ have the same value $\tilde{d}(v)$, we can order them by their value of $\tilde{D}(v)$. Below, we show that we can switch $v_i$ and $w_j$ in the order $\prec$. By repeatedly doing these kind of switches, we can make sure that $\tilde{D}(v_i)=\tilde{d}(v_i)-1$ for all $i$.

After switching $v_i$ with $w_j$, the value of $\tilde{d}(v_i)$ decreases by one and $\tilde{d}(w_j)$ increases by one. If $\sum_{p\in (v_i,w_j)^{\circ}}\,D(p)=1$, the same happens for $\tilde{D}(v_i)$ and $\tilde{D}(w_j)$, while if $\sum_{p\in (v_i,w_j)^{\circ}}\,D(p)=0$, the values $\tilde{D}(v_i)$ and $\tilde{D}(w_j)$ remain unchanged. Hence, we will still get that $\tilde{D}(v_i)<\tilde{d}(v_i)$ and $\tilde{D}(w_j)<\tilde{d}(w_j)$, so we are allowed to do the switch.
\end{proof}

If $D(v_m) < r$, we can already conclude that the rank of the divisor cannot be $r$, so assume $D(v_m) \geq r$.

We write
\[D(v_m) = An + B, \ \text{with} \ A, B \in \mathbb{Z}_{\geq 0} \ \text{and} \ 0 \leq B \leq n-1.\]
By chip-firing $A+1$ or $A$ times from $v_m$, we obtain the divisors:
\begin{align*} D_1 = (B-n)(v_m) & + \sum_{i=1}^n \Big[ \big( D(w_i) + A +1\big)(w_i) +  \sum_{v_j \prec w_i} \sum_{p\in (w_i,v_j)^{\circ}}\,D(p) \,(p)\Big] \\
 & + \sum_{i=1}^{m-1} \Big[D(v_i) (v_i) + \sum_{w_j \prec v_i} \sum_{p\in (v_i,w_j)^{\circ}}\,D(p) \,(p)\Big],
\end{align*}
\begin{align*}D_2 = B(v_m) &+ \sum_{i=1}^n \Big[ \big( D(w_i) + A\big)(w_i) +  \sum_{v_j \prec w_i} \sum_{p\in (w_i,v_j)^{\circ}}\,D(p) \,(p)\Big] \\
&+ \sum_{i=1}^{m-1} \Big[D(v_i) (v_i) +  \sum_{w_j \prec v_i} \sum_{p \in (v_i,w_j)^{\circ}}\,D(p) \,(p)\Big].
\end{align*}

We can conclude that $\rk(D)<r$ if we are able to construct an effective divisor $E_1$ that satisfies the following conditions:
\begin{itemize}
\item the degree of $E_1$ is at most $r=(a+1)(b+1) - h -1$,
\item $E_1$ is supported on the vertices $w_i$ and on the points $p$ of $\supp(D)$ with $p \in (w_i, v_j)^{\circ}$ for $v_j \prec w_i$, and,
\item $D_1 - E_1$ is $v_m$-reduced with respect to the same ordering $\prec$,
\end{itemize}
or, if we can construct an effective divisor $E_2$ that satisfies:
\begin{itemize}
\item the degree of $E_2$ is at most $r=(a+1)(b+1) - h -1$,
\item $E_2$ is supported on the sink $v_m$, on the vertices $w_i$ and on the points $p$ of $\supp(D)$ with $p \in (w_i, v_j)^{\circ}$ for $v_j \prec w_i$,
\item the coefficient of $E_2$ at $v_m$ is $B+1$, and,
\item $D_2 - E_2$ is $v_m$-reduced with respect to the same ordering $\prec$.
\end{itemize}
Indeed, if such a divisor $E_i$ can be constructed (with $i\in\{1,2\}$), then the divisor $D_i-E_i$ is $v_m$-reduced and has a negative coefficient at $v_m$. Therefore, by the definition of reduced divisors, there is no effective divisor equivalent with $D_i-E_i\sim D - E_i$, hence $\rk(D)<\deg(E_i)\leq r$.

We claim that it is always possible to construct an $E_1$ or an $E_2$ satisfying the conditions. In other words, when we add the principal divisor
\[
 -n (v_m) + \sum_{i = 1}^{n} (w_i)
\]
to $D$ until it either has a negative value at $v_m$, or one time before that, then (at least) one of the two resulting divisors is $v_m$-reduced after subtracting an effective divisor of degree at most $r=(a+1)(b+1)-h-1$.

\vspace{\baselineskip}

For notational purposes, set $\alpha_i := \tilde{D}(w_i)  + A - (\tilde{d}(w_i) -2)$ for each $i\in\{1,\ldots,n\}$ and write $x^+:=\max\{x,0\}$ for any $x\in\RR$. Then $E_1$ can be constructed if
\[
\sum_{i=1}^n \alpha_i^+ \ \leq\  (a+1)(b+1) - h -1,
\]
while $E_2$ can be constructed if
\[
B + 1 + \sum_{i=1}^n (\alpha_i-1)^+ \ \leq\ (a+1)(b+1) - h -1.
\]

The sequence $(\alpha_i)_{i=1,\ldots,n}$ satisfies $A+2-m\leq \alpha_i \leq A+1$ and $\alpha_1=A+1$.
Moreover, there is a formula for the sum $\sum_{i=1}^n \alpha_i$. Indeed,
since $$mn = \sharp(E(K_{m,n})) = \sum_{v\in V(K_{m,n})}\, \tilde{d}(v) = \sum_{i=1}^n \tilde{d}(w_i) + \sum_{i<m} \tilde{d}(v_i),$$ it follows that
\[
\begin{aligned}
\sum_{i=1}^{n} \alpha_i &= \sum_{i=1}^n \tilde{D}(w_i) - \sum_{i=1}^n \tilde{d}(w_i) + n(A+2) \\
&=  \sum_{i=1}^n \tilde{D}(w_i) + \sum_{i<m} \tilde{d}(v_i) -nm + n(A+2) \\
&=  \sum_{i=1}^n \tilde{D}(w_i) + \sum_{i<m} \tilde{D}(v_i) + (m-1) -nm + n(A+2) \\
&= \textrm{deg}(D) - D(v_m) + (m-1) - nm + n(A+2) \\
&= (an + bm -h -1) -(An +B) + (m-1) - nm + n(A+2) \\
&= an + bm - h +m + 2n -nm - B - 2.
\end{aligned}
\]

\begin{defi} Given integer parameters $m,n,a,b,h,A,B$, we say that the sequence $\alpha=(\alpha_i)_{i=1,\ldots,n}$
satisfies the conditions $(*)$ if and only if
$$A+2-m\leq \alpha_i \leq A+1$$ and $$\sum_{i=1}^{n} \alpha_i= an + bm - h +m + 2n -nm - B - 2.$$
\end{defi}

Hence, in order to show that $\rk(D)<r$, it is sufficient to prove the following claim, which rephrases the existence of an effective divisor $E_1$ or $E_2$ into some property of the sequence $(\alpha_i)_{i=1,\ldots n}$.

\begin{claim} \label{claim}
Consider integers $m,n,a,b,h,A,B\geq 0$ such that
$$r:=(a+1)(b+1)-1-h \ ,\ 0<r<(m-1)(n-1)\ ,$$
$$a\leq m-1\ ,\ b\leq n-1\ ,B\leq n-1$$
and such that $(a,b)$ attains the maximal value of $(m-a-1)(n-b-1)$.
Let $\alpha = (\alpha_i)_{i=1,\ldots n}$ be a sequence of integers satisfying the conditions $(*)$ and
define
\[ t_1 := \sum_{i=1}^n \alpha_i^+ \ \text{and} \ t_2 := B + 1 + \sum_{i=1}^n \big(\alpha_i -1\big)^+.\]
Then $\min \{t_1, t_2\} \leq r$.
\end{claim}

\begin{lem} \label{lem_easycase} If $A+2-m >0$, then Claim \ref{claim} holds.
\end{lem}

\begin{proof} Remark that $(\alpha_i -1)^+ = \alpha_i -1$ for every $i$. We are going to show that $t_2 \leq r$. First, let's compute $t_2$:
\begin{equation*}
t_2 = \sum_{i=1}^n (\alpha_i -1)^+ + B + 1 = \sum_{i=1}^n \alpha_i - n + B + 1 = an + bm - h +m + n -nm - 1
\end{equation*}
So we have that
\begin{align*}
r-t_2 &= \left((a+1)(b+1) - h -1\right) - \left(an + bm - h +m + n -nm - 1\right) \\
& = (n-a-1)(n-b-1).
\end{align*}
By hypothesis $a\leq m-1$ and $b\leq n-1$, hence $r-t_2 \geq 0$.
\end{proof}

From now on, we add the hypothesis $A+2-m \leq 0$. To prove the claim, we will proceed as follows: first we introduce a specific integer sequence, which we show to be the ``worst-case scenario". Afterwards, we prove the claim for this particular sequence. \\

For each $p \in \{1, \ldots, n\}$ and $q \in \{A+2-m, \ldots, A\}$, we define the sequence $\beta^{(p,q)} = \big(\beta_i \big)$ as follows:
\begin{equation} \tag{$\triangle$}  \label{sequence}  \beta_1 = \cdots = \beta_p = A+1, \ \ \ \beta_{p+1} = q, \ \ \ \beta_{p+2}= \ldots =\beta_{n} = A+2-m.
\end{equation}
We want that the sequence $\beta^{(p,q)}$ satisfies the same conditions $(*)$ as the sequence $\alpha$. In particular, we need that
\[\sum_{i=1}^n \beta_i = an + bm - h +m + 2n -nm -B -2.
\]
This equation allows us to compute $p$ and $q$. Indeed, since
\begin{align*}
\sum_{i=1}^n \beta_i &= (A+1)p + q + (A+2-m)(n-p-1) \\
&= p(m-1) + \big( q - (A+2-m)\big) + n (A+2-m),
\end{align*}
it follows that
\begin{equation} \tag{$\Diamond$} \label{pandq} p(m-1) + \big( q - (A+2-m)\big) = an + bm - h + m -An -B -2. \end{equation}
Since $0\leq q -(A+2-m) \leq m-2$, the parameters $p$ and $q$ are uniquely determined by Euclidean division of the right hand side of the equation \eqref{pandq} by $m-1$.

\begin{lem} \label{worstcase} Assume that $A+2-m\leq 0$ and
let $\alpha = (\alpha)_{i=1, \ldots, n}$ be a sequence of integer numbers that satisfies the conditions $(*)$.
If $p,q$ are the integers such that the sequence $\beta^{(p,q)}$ satisfies $\sum_{i=1}^n \beta_i = \sum_{i=1}^n \alpha_i$, then
\[ \sum_{i=1}^n \alpha_i^+ \leq \sum_{i=1}^n \beta_i^+ \ \textrm{ and } \ \sum_{i=1}^n (\alpha_i-1)^+ \leq \sum_{i=1}^n (\beta_i-1)^+ \]
\end{lem}

\begin{proof}
If $\sum_{i=1}^n \alpha_i^+ > \sum_{i=1}^n \beta_i^+$, then $\sharp\{i\,|\,\alpha_i>0\}\geq p+1$ since $\beta_1,\ldots,\beta_p$ reach the maximal value $A+1>0$. This implies that $$\sum_{i=1}^n \alpha_i\geq \sum_{i=1}^n \alpha_i^+ + (n-p-1)(A+2-m)>\sum_{i=1}^n \beta_i^+ + (n-p-1)(A+2-m)\geq \sum_{i=1}^n \beta_i,$$
a contradiction. Similarly, we can see that $$\sum_{i=1}^n (\alpha_i-1)^+ \leq \sum_{i=1}^n (\beta_i-1)^+.$$
\end{proof}

Now we are able to prove Claim \ref{claim}.

\begin{proof} Because of Lemma \ref{worstcase}, the sequence $\beta^{(p,q)}$ defined by (\ref{sequence}) maximizes $\sum_{i=1}^n \beta_i^+$ and $\sum_{i=1}^n (\beta_i -1)^+$. Therefore it is enough to check our claim for this kind of sequence. Recall that we may assume that $A+2-m \leq 0$ by Lemma \ref{lem_easycase}. We distinguish the following four cases:
\begin{enumerate}
\item $q>0$ and $B<p$; \label{case1}
\item $q>0$ and $B \geq p$; \label{case2}
\item $q \leq 0$ and $B<p$; \label{case3}
\item $q \leq 0$ and $B \geq p$. \label{case4}
\end{enumerate}

We will only handle the cases (\ref{case1}) and (\ref{case2}); the cases (\ref{case3}) and (\ref{case4}) can be treated in a similar way.
Below, we use the following equality which directly follows from (\ref{pandq}):
\begin{gather} 
\tag{$\Diamond'$} \label{pandq'} an + bm - h = An + pm + \big(q - (A+2-m) - (m-2) \big) + (B-p).
\end{gather}

\vspace{\baselineskip}
{\bf Case (\ref{case1})}. We are going to show that $r-t_2 \geq 0$. Since
\[t_2 = \sum_{i=1}^n (\beta_i -1)^+ + B + 1 = Ap + (q-1) + B + 1 = Ap + q + B. \]
and by using (\ref{pandq}), we obtain that
$$r-t_2 = (m-a-1)(n-b-1) - (m-A-1)(n-p-1).$$
Moreover, from (\ref{pandq'}), by using that $p<B$ and $q- (A+2-m) \leq m-2$, we get that
\begin{equation} \label{ineq1} an + bm - h \leq pm + An -1. \end{equation}
Now suppose that $r-t_2 <0$, which is equivalent to
\begin{equation} \label{ineq1b} (m-a-1)(n-b-1) < (m-A-1)(n-p-1). 
\end{equation}
By adding the inequalities (\ref{ineq1}) and (\ref{ineq1b}), we have that
$$ (a+1)(b+1) - h - 1 \leq (A+1)(p+1) -3, $$
hence $(A+1)(p+1)-1-r \geq 0$. Now Remark \ref{condition} implies
\[
(m-a-1)(n-b-1) \geq (m-A-1)(n-p-1),
\]
which contradicts our assumption.

\vspace{\baselineskip}
{\bf Case (\ref{case2})}. In this case, we want to show that $r-t_1 \geq 0$. Therefore, we first compute $t_1$:
\[t_1 = \sum_{i=1}^n \beta_i^+ = (A+1)p + q. \]
Using (\ref{pandq}), we find that
\[r-t_1 = (m-a-1)(n-b-1) - (m-A-1)(n-p-1) + B-p.\]
From (\ref{pandq'}), by using that  $q- (A+2-m) \leq m-2$, we obtain
\begin{equation} \label{ineq2} an + bm - h \leq p(m-1) + An +B. \end{equation}
Suppose that $r-t_1<0$, which means
\begin{equation} \label{ineq2b}
(m-a-1)(n-b-1) < (m-A-1)(n-p-1) - (B-p) 
\end{equation}
By adding the inequalities (\ref{ineq2}) and (\ref{ineq2b}), we obtain that
$$ 
(a+1)(b+1) - h - 1 \leq (A+1)(p+1) -1, $$
thus $(A+1)(p+1)-1-r \geq 0$. By Remark \ref{condition}, this implies
\[
(m-a-1)(n-b-1) \geq (m-A-1)(n-p-1),
\]
which contradicts our assumption since $B\geq p$.

\end{proof}

\begin{cor} \label{cor_generic}
Let $C$ be a generic smooth curve of bidegree $(m,n)$ on $\mathbb{P}^1\times\mathbb{P}^1$. Then $d_r(C)=\delta_r(m,n)$ for all
$0<r<g(C)=(m-1)(n-1)$.
\end{cor}

\begin{proof}
As in Remark \ref{rem_phi}, consider a bivariate polynomial $$f = \sum_{i=0}^m \sum_{j=0}^n a_{i,j} x^i y^j\in k[x,y]$$ defining a smooth curve of bidegree $(m,n)$. Moreover, consider distinct lines $L_{1,i}\in R_1$ with $i\in\{1,\ldots,m\}$, defined by $\ell_{1,i}:=x-x_i=0$, and $L_{2,j}\in R_2$ with $j\in\{1,\ldots,n\}$, defined by $\ell_{2,j}:=y-y_j=0$.
The equation
\[ t \cdot f + \prod_{i=1}^m  \ell_{1,i} \prod_{j=1}^n \ell_{2,j}\in k[[t]][x,y] \]
defines a $1$-parameter family of smooth curves of bidegree $(m,n)$. Its generic fiber is a smooth curve over $k((t))$. The family degenerates to the union of the $m+n$ lines $L_{1,i}$ and $L_{2,j}$, whose dual graph is the complete bipartite graph $K_{m,n}$ with edge lengths
$l(e)=1$. By Baker's Specialization Lemma \cite[Lemma 2.8]{Bak}, we know that linear systems on the generic fiber of the family specialize to linear systems on the graph $K_{m,n}$, and that the rank can only increase under specialization.
Therefore, we obtain the following inequality for a generic bidegree-$(m,n)$ curve $C$ over $k$ (by semi-continuity of the gonality sequence \cite[Proposition 3.4]{LM}):
$$d_r(K_{m,n}) \leq  d_r(C).$$
By Theorem \ref{thm_Kmn}, the left hand side of this inequality equals $\delta_r(m,n)$. On the other hand, by Lemma \ref{lem_upperbound}, the right hand side is at most $\delta_r(m,n)$, so the statement follows.
\end{proof}

\section{Sharpness for smooth curves on $\PP^1 \times \PP^1$}

In this last section, we prove Theorem \ref{thm_main}.
We start with a brief outline of the argument that Hartshorne \cite{Har} used to compute the gonality sequence of plane curves (see Theorem \ref{thm_planecurves}).  We then adapt his argument to curves on $\mathbb{P}^1\times \mathbb{P}^1$.

Consider a smooth plane curve $C\subset \mathbb{P}^2$. Hartshorne's proof proceeds by induction on the degree $d$ of $C$. If a divisor $D$ on $C$ is non-special, then its rank can be computed via Riemann-Roch, namely $\rk_C(D)=g+r$ with $g=\frac{(d-1)(d-2)}{2}$. If instead the divisor $D$ is special, then it is contained in a plane curve $C'$ of degree $d-3$, and one can derive a formula for its rank as a divisor on $C$ in terms of its rank as a divisor on $C'$. The main issue with this argument is that there is no guarantee that the curve $C'$ must be smooth. For this reason, Hartshorne developed the theory of \emph{generalized divisors} on Gorenstein curves, see \cite[Section 1]{Har} for further details. This approach has a secondary advantage: if one restricts to generalized divisors, one does not need to impose the smoothness condition on the curve $C$.  In fact, Hartshorne only assumes that $C$ is irreducible, see \cite[Theorem 2.1]{Har}.

\begin{exam}
Let $C$ be an irreducible plane curve of degree $d\geq 3$ with a node at $P$. Then the projection map $f:C\to \mathbb{P}^1$ from $P$ is non-degenerate and rational of degree
$d-2$. However, the corresponding divisor $D=H\cap C-2P$, where $H\subset \mathbb{P}^2$ is a line, is not a generalized divisor (see \cite[Example 1.6.1]{Har}). In fact, Hartshorne shows that $C$ still satisfies $d_1(C)=d-1$.
\end{exam}

A natural question is whether Hartshorne's argument can be adapted to curves on some other surface $S$.  Note that, if $S$ is smooth, then any curve on $S$ is Gorenstein.  If any multiple of the canonical bundle $K_S$ is effective, then the inductive procedure will not terminate, so we should assume that $S$ is rational or ruled. This in itself is not much of a restriction: indeed, any curve can be embedded in a rational surface.
However, at a crucial step in our argument we use the fact that, for any two effective curve classes $C$ and $F$ on $S$, the restriction $F\vert_C$ is effective. The only rational surfaces with this property are $\PP^2$ and $\PP^1 \times \PP^1$.

Throughout, if $D$ is a divisor on a curve $C$ on $S=\PP^1 \times \PP^1$, we write $D+(x,y)$ for the divisor class $D + \OO_C(x,y)$. Our main result is the following.

\begin{thm}
\label{thm:mainthm}
Let $C$ be an irreducible curve of bidegree $(m,n)$ on $S = \PP^1 \times \PP^1$.  Let $D$ be a (generalized) divisor on $C$ of rank $r \geq 0$ and degree $d \leq 2(mn-m-n)$.  Then we have that
\[
\deg (D) \geq \min\{am + bn - h\,|\,(a,b,h) \in I_r\}=\delta_r(m,n).
\]
Moreover, if $D$ is of the form $D = D'+(x,y)$ for some effective $D'$ and $x,y \geq 0$, then the minimum in this expression can be taken over all $(a,b,h) \in I_r$ such that $a \geq y$ and $b \geq x$.
\end{thm}

In other words, the divisors of smallest degree for a given rank are simply restrictions of line bundles from the ambient $\PP^1 \times \PP^1$, minus base points.

We start with the following observation.

\begin{lem}
\label{Lem:Surj}
Let $C$ be a curve on a smooth Fano surface $S$.  The restriction map
\[
H^0 (S, K_S + C) \to H^0 (C, K_C)
\]
is an isomorphism.
\end{lem}

\begin{proof}
Note that, by adjunction, $K_C = (K_S +C)\vert_C$.  Now, consider the long exact sequence on cohomology
\[
H^0 (S, K_S) \to H^0 (S, K_S +C) \to H^0 (C, K_C ) \to H^1 (S, K_S) .
\]
Since $K_S$ is anti-ample, $H^0 (S, K_S) = 0$.  By Serre duality, $H^1 (S, K_S ) \cong H^1 (S, \OO_S)$, and $H^1 (S, \OO_S) = 0$ because Fano varieties are simply connected.  It follows that the center arrow is an isomorphism.
\end{proof}

In other words, the canonical linear system on $C$ is cut out by restrictions of curves in the class $\vert K_S + C \vert$ to $C$.

We now prove an analogue of \cite[Lemma 2.2]{Har} for curves on $\PP^1 \times \PP^1$. This lemma, which allows us to compute the ranks of divisors on curves in terms of the ranks of related divisors on curves of smaller bidegree, provides the key step in our inductive argument. In what follows, when $X$ is clear from the context, we will use the notation $H^i(\mathcal{F})$ to indicate the cohomology space $H^i(X, \mathcal{F})$ of some sheaf $\mathcal{F}$. Similarly, the dimension $h^i(X, \mathcal{F})$ will be abbreviated to $h^i(\mathcal{F})$.

\begin{lem}
\label{Lem:InductiveStep}
Let $C$ be an irreducible curve of bidegree $(m,n)$ on $S = \PP^1 \times \PP^1$ and $Z$ a closed subscheme of finite length of $C$.  Let $F$ be an effective curve of bidegree $(e,f)$ with $e \geq f$ and $Z' = Z + C \cap F$.  Suppose that $Z$ is contained in an irreducible curve $C'$ of bidegree $(m',n')$ with
\begin{displaymath}
(m',n') = \left\{ \begin{array}{ll}
(m-1,n) & \textrm{if $e \neq 0$},\\
(m-1,n-1) & \textrm{if $e=0$}
\end{array} \right.
\end{displaymath}
Then either $Z'$ is non-special or

\begin{multline*}
h^0 \big(\LL_C (Z')\big) = (e+1)(f+1) - (m'+e-m+1)(n'+f-n+1) \\
+ h^0 \big(\LL_{C'} (Z+(m'+e-m,n'+f-n))\big) .
\end{multline*}
\end{lem}

\begin{proof}
By Serre duality we have
\[
h^0 \big(\LL_C (Z')\big) = h^1 \big(\LL_C (K_C - Z')\big) = h^1 \big(\LL_C (K_C - Z - F)\big) = h^1 \big(\II_{Z,C} (K_C - F)\big) .
\]
By adjunction, this last term can be written as
\[
h^1 \big(\II_{Z,C} (K_C - F)\big) = h^1 \big(\II_{Z,C} (K_S +C-F)\big) .
\]
We have the exact sequence
\[
0 \to \OO_S (-C) \to \II_Z \to \II_{Z,C} \to 0,
\]
which yields an exact sequence on cohomology
\[
H^1 \big(\OO_S (K_S - F)\big) \to H^1 \big(\II_Z (K_S +C-F)\big) \to H^1 \big(\II_{Z,C} (K_S +C-F)\big) \to
\]
\[
\to H^2 \big(\OO_S (K_S - F)\big) \to H^2 \big(\II_Z (K_S +C-F)\big) \to 0 .
\]
Now, if $F \cap C-K_C$ is an effective divisor on $C$, then $\LL_C (Z')$ is non-special.  Note that every effective divisor on $S$ is basepoint free, and therefore if $F-K_S-C$ is an effective divisor on $S$, then $F \cap C-K_C$ is an effective divisor on $C$.  Hence we may assume that $F-K_S-C$ is non-effective on $S$.

We first show that $H^2 \big(\II_Z (K_S +C-F)\big) = 0$.  We have the exact sequence
\[
H^1 \big(\OO_Z (K_S +C-F)\big) \to H^2 \big(\II_Z (K_S +C-F)\big) \to
\]
\[
\to H^2 \big(\OO_S (K_S +C-F)\big) \to H^2 \big(\OO_Z (K_S +C-F)\big) .
\]
Because $Z$ is zero dimensional, both the left and right terms vanish, and we have
\[
H^2 \big(\II_Z (K_S +C-F)\big) \cong H^2 \big(\OO_S (K_S +C-F)\big).
  \]
By Serre duality,
\[
h^2 \big(\OO_S (K_S +C-F)\big) = h^0 \big(\OO_S (F-C)\big).
\]
Since $-K_S$ is ample, if $F-C$ is effective on $S$ then $F-K_S -C$ would be effective as well.  Since this is not the case, we see that $h^0 \big(\OO_S (F-C)\big) = 0$.

We now show that $H^1 \big(\OO_S (K_S - F)\big) = 0$.  By Serre duality,
\[
h^1 \big(\OO_S (K_S - F)\big) = h^1 \big(\OO_S (F)\big) .
\]
Since $-K_S$ is ample and $F$ is effective, $F-K_S$ is big.  Every big divisor on $S$ is nef, therefore we see that $h^1 \big(\OO_S (F)\big) = 0$ by Kawamata-Viewheg vanishing.

From this, we conclude that
\[
h^1 \big(\II_{Z,C} (K_S +C-F)\big) = h^1 \big(\II_Z (K_S +C-F)\big) + h^2 \big(\OO_S (K_S - F)\big) .
\]
We now run the same arguments on $C'$.  We see that, again,
\[
h^1 \big(\II_{Z,C'} (K_S + C-F)\big) = h^0 \big(\LL_{C'} (Z+F-(C-C'))\big) .
\]
and we have the exact sequence
\[
H^1 \big(\OO_S (K_S +C-C'-F)\big) \to H^1 \big(\II_Z (K_S +C-F)\big) \to H^1 \big(\II_{Z,C'} (K_S +C-F)\big) \to
\]
\[
\to H^2 \big(\OO_S (K_S +C-C'-F)\big) \to H^2 \big(\II_Z (K_S +C-F)\big) \to 0 .
\]
As above, we have $H^2 \big(\II_Z (K_S +C-F)\big)=0$.  To see that $H^1 \big(\OO_S (K_S +C-C'-F)\big) = 0$, note that if $e \neq 0$, then $C'+F-C$ is effective, and we obtain the vanishing for the same reason as above.  If $e=0$, then $C'+F-C$ has bidegree $(-1,-1)$, and $H^1 \big(\OO_S (-1,-1)\big) = 0$.  It follows that
\[
h^1 \big(\II_{Z,C'} (K_S +C-F)\big) = h^1 \big(\II_Z (K_S + C-F)\big) + h^2 \big(\OO_S (K_S - F + (C-C'))\big).
\]

Putting this all together, we see that

\noindent $h^0 (\LL_C (Z'))$
\begin{multline*}
 = h^0 \big(\LL_{C'} (Z+F-(C-C'))\big) + h^2 \big(\OO_S (K_S - F)\big) \\
- h^2 \big(\OO_S (K_S - F + (C-C'))\big)
\end{multline*}
\begin{multline*}
= h^0 \big(\LL_{C'} (Z+F-(C-C'))\big)  + h^0 \big(\OO_S (F)\big) - h^0 \big(\OO_S (F-(C-C'))\big)
\end{multline*}
\begin{multline*}
= (e+1)(f+1) - (m'+e-m+1)(n'+f-n+1) \\
+ h^0 \big(\LL_{C'} (Z+(m'+e-m,n'+f-n))\big).
\end{multline*}\\
\end{proof}

We now prove the main theorem.

\begin{proof}[Proof of Theorem~\ref{thm:mainthm}]
We prove this by induction on the bidegree $(m,n)$.  If either $m$ or $n$ is at most 1, then the result is trivial.  We henceforth assume that both $m$ and $n$ are at least 2, and the result holds for all curves of bidegree $(m',n')$ with $m' \leq m$, $n' \leq n$, not both equal.  

We may assume that $D$ is special.  By Lemma \ref{Lem:Surj} there exist curves, possibly reducible and singular, of bidegree $(m-2,n-2)$ containing $D$.  Let $\vert \II_D (m-2,n-2) \vert$ denote the linear system of all curves on $S$ of bidegree $(m-2,n-2)$ containing $D$.  Let the fixed component of this linear system be a curve $F$ of bidegree $(e,f)$.  Without loss of generality, we assume that $e \geq f$.  Let $D'$ be the scheme-theoretic union of $D$ and $C \cap F$.  Then $D'$ is a closed subscheme of $C$ of degree $d' \geq d$.  By construction, all curves of bidegree $(m-2,n-2)$ containing $D$ contain $D'$, so $h^0 \big( \II_D (m-2,n-2)\big) = h^0 \big( \II_{D'} (m-2,n-2)\big)$.  This implies that $h^1 \big( \LL_C (D)\big) = h^1 \big( \LL_C (D')\big)$.  By Riemann-Roch, therefore, $h^0 \big( \LL_C (D)\big) = h^0 \big( \LL_C (D')\big) - (d'-d)$.  It therefore suffices to bound $h^0 \big( \LL_C (D')\big)$.

Now let $D'' = D' - C \cap F$.  We have an exact sequence
\[
0 \to \II_{D''} (-e,-f) \to \II_{D'} \to \II_{C \cap F, F} \to 0 .
\]
Twisting and taking cohomology, this yields an exact sequence
\[
0 \to H^0 \big(\II_{D''} (m-e-2,n-f-2)\big) \to H^0 \big( \II_{D'} (m-2,n-2)\big)
\]
\[
\to H^0 \big( \II_{C \cap F, F} (m-2,n-2)\big) .
\]
Since $F$ is the fixed component of $\vert \II_{D'} (m-2,n-2) \vert$, the rightmost map is zero.  It follows that the linear system $ \vert \II_{D''} (m-e-2,n-f-2) \vert$ has no fixed component.  If $e \neq 0$, then adding arbitrary curves of bidegree $(e+1,f+2)$, we see that $\vert \II_{D''} (m-1,n) \vert$ has no fixed component.  By one of the Bertini theorems \cite[p.30]{Zar}, the general member of $\vert \II_{D''} (m-1,n) \vert$ is irreducible.  If $e=0$, then by a similar argument we see that $\vert \II_{D''} (m-1,n-1) \vert$ has no fixed component and its general member is irreducible.

We first consider the case where $e=0$.  By the above, $D''$ is contained in an irreducible curve $C'$ of bidegree $(m-1,n-1)$.  Then, by Lemma \ref{Lem:InductiveStep}, we have
\[
h^0 \big(\LL_C (D')\big) = h^0 \big(\LL_{C'} (D''+(-1,-1))\big)+1 .
\]
If $D''+(-1,-1)$ is non-effective, there is nothing to prove.  If $D''+(-1,-1)$ is effective and non-special on $C'$, then
\[
h^0 \big(\LL_{C'} (D''+(-1,-1))\big) = d'-m-n+2-(m-2)(n-2)+1,
\]
so
\[
h^0 \big(\LL_C (D')\big) = d'-m-n-(m-2)(n-2)+4 = d'-(m-1)(n-1)+1,
\]
so $D'$ is non-special on $C$.  Otherwise, if $D''+(-1,-1)$ is effective on $C'$, then by induction we have
\[
d'-m-n+2 = \deg \big(\LL_{C'} (D''+(-1,-1))\big) \geq \min_{(a,b,h) \in I_{r-1}} a(m-1) + b(n-1) - h .
\]
In other words, there exist constants $a,b$, and $h$ such that
\[
d' \geq a(m-1) + b(n-1) - h+m+n-2
\]
and
\[
(a+1)(b+1) - h = r.
\]
Letting $a'=a+1$, $b'=b+1$, and $h'=h+a+b+2$, we see that
\[
d' \geq a'm+b'n-h'
\]
and
\[
(a'+1)(b'+1) - h' = (a+1)(b+1)-h + 1 = r+1.
\]
Hence
\[
d' \geq \min_{(a,b,h) \in I_r} am + bn - h.
\]

We now consider the case where $e \neq 0$.  In this case, $D''$ is contained in an irreducible curve $C'$ of bidegree $(m-1,n)$.  By Lemma \ref{Lem:InductiveStep}, we have
\[
h^0 \big(\LL_C (D')\big) = h^0 \big(\LL_{C'} (D''+(e-1,f))\big) +f+1 .
\]
Note that $D''+(e-1,f)$ is effective, since $e>0$.  If $D''+(e-1,f)$ is non-special on $C'$, then
\[
h^0 \big(\LL_{C'} (D''+(e-1,f))\big) = d'-n-f-(m-2)(n-1)+1,
\]
so
\[
h^0 \big(\LL_C (D')\big) = d'-n-(m-2)(n-1)+2 = d'-(m-1)(n-1)+1,
\]
and $D'$ is non-special on $C$.

Otherwise, by induction we have
\[
d'-n-f = \deg \big(D''+(e-1,f)\big) \geq \min_{(a,b,h) \in I_{r-f-1}} a(m-1) + bn - h .
\]
In other words, there exist constants $a,b$, and $h$ such that
\[
d' \geq a(m-1) + bn - h+n+f
\]
and
\[
(a+1)(b+1) - h = r-f.
\]
Note that $a \geq f$, hence $h+a-f \geq 0$.  Let $a'=a$, $b'=b+1$, and $h'=h+a-f$.  Then
\[
d' \geq a'm+b'n-h'
\]
and
\[
(a'+1)(b'+1) - h' = (a+1)(b+1)-h + (f+1) = r+1.
\]
Hence
\[
d' \geq \min_{(a,b,h) \in I_r} am + bn - h .
\]

Finally, notice that in both cases our choice for $a'$ and $b'$ do not decrease, so the inequalities $a \geq y, b \geq x$ follow.
\end{proof}

\end{document}